\documentclass[a4paper,12pt]{amsart}

\usepackage{amsmath,amssymb,amsthm}
\usepackage{amsfonts}
\usepackage{multicol}
\usepackage{color}

\newtheorem{theorem}{Theorem}[section]
\newtheorem{lemma}[theorem]{Lemma}
\newtheorem{e-proposition}[theorem]{Proposition}
\newtheorem{corollary}[theorem]{Corollary}
\newtheorem{e-definition}[theorem]{Definition\rm}

\newcommand{\charI}{1 \hspace*{-1mm} {\rm l}}
 
\def\ignore#1{}
\def\eq{\begin{equation}}
\def\en{\end{equation}}
\def\eqa{\begin{eqnarray}}
\def\ena{\end{eqnarray}}
\def\eqs{\begin{eqnarray*}}
\def\ens{\end{eqnarray*}}
\def\bZ{{\mathbb Z}}
\def\bP{{\mathbb P}}
\def\bE{{\mathbb E}}
\def\re{{\mathbb R}}
\def\bPstar{\bP_{i,s^*}}
\def\non{\nonumber}
\def\rstar{\rho_{s^*}}
\def\sjmo{\sum_{j\in\bZ\setminus\{0\}}}
\def\intnc{\lfloor nc \rfloor}
\def\tg{{\tilde g}}
\def\s{\sigma}
\def\l{\lambda}
\def\d{\delta}
\def\Ref#1{(\ref{#1})}
\def\hP{{\widehat P}}
\def\Eq{\ =\ }
\def\Le{\ \le\ }
\def\Z{{\bf Z}}
\def\h{\eta}
\def\t{\tau}
\def\n{\nu}
\def\e{\varepsilon}
\def\f{\phi}
\def\a{\alpha}
\def\hl{{\hat \l}}
\def\gg{{\mathcal G}}
\def\nin{\noindent}
\def\tV{{\tilde V}}

\begin{document}

\title[Translated Poisson approximation]{Translated Poisson approximation to \\ equilibrium distributions of Markov
population processes}

\author[S.\ Socoll \& A. D. Barbour]{Sanda N.\ Socoll \& A. D. Barbour}

\address{Institut f\"ur Mathematik, Universit\"at Z\"urich-Irchel \\ Winterthurerstr. 190, 8057 Z\"urich \\ Switzerland}

\begin{abstract}
The paper is concerned with the equilibrium
distributions of continuous-time density dependent Markov processes on the integers. 
These distributions
are known typically to be approximately normal, with $O( 1 /{\sqrt{n}})$ error as 
measured in
Kolmogorov distance. Here, an approximation in the much stronger total variation norm 
is established, 
without any loss in the asymptotic order of accuracy; the approximating distribution
is a translated Poisson distribution having the same variance and (almost) the same mean. 
Our arguments are based on the Stein-Chen method and Dynkin's formula. 
\end{abstract} 



\thanks{This work was supported in part by the Swiss National Foundation, Project No.~200020-107935/1;
some of it was undertaken while the authors were visiting the Institute for Mathematical Sciences, 
National University of Singapore.}

\maketitle{}

\setcounter{equation}{0}
\section{Introduction}

Density dependent Markov population processes, in which the transition rates depend
on the density of individuals in the population, have proved widely useful as models in the social and life
sciences: see, for example, the monograph of Kurtz~(1981), in which approximations in terms of
diffusions are extensively discussed, in the limit as the typical population size~$n$ tends to
infinity. Here, we are interested in the behavior at equilibrium.
Our starting point is the paper of Barbour~(1980), in which conditions are given 
for the existence of an equilibrium distribution concentrated close to the 
deterministic equilibrium, together with a bound of
order $O(1/\sqrt{n})$ on the Kolmogorov distance between the equilibrium distribution and a suitable
normal distribution. 
We now show that this normal approximation can be substantially strengthened. Using a delicate argument based
on the Stein--Chen method, we are able to establish an approximation in total variation in terms of a
translated Poisson distribution. What is more, our error bounds with respect to this much stronger metric, 
and under weaker assumptions than those previously considered, are still of ideal order 
$O(1/\sqrt{n})$. 

The first step in the argument is to establish the existence of an equilibrium distribution under
suitable conditions, and to show that it is appropriately concentrated around the `deterministic'
equilibrium, defined to be the stationary point of an associated system of differential equations
which describe the average drift of the process in the limit as $n\to\infty$; this is
accomplished in Section~\ref{equilibrium}.  The closeness of
this distribution to our approximation is then established in Section~\ref{main}, by showing 
that Dynkin's formula, applied in equilibrium, yields an equation not far removed from the Stein equation
for a centred Poisson distribution, enabling ideas related to Stein's method to be
brought into play.  An important element in obtaining an approximation in total variation is to
show {\it a priori\/} that the equilibrium distribution is sufficiently smooth, in the sense
that translating it by a single unit changes the distribution only by order $O(1/\sqrt{n})$ in total
variation: see, for example, R\"ollin~(2005). The corresponding argument is to be found in 
Section~\ref{one-shift}. We illustrate the results by applying them to a birth, death 
and immigration process, with births occurring in groups.

\subsection{Basic approach}\label{prelims}

We start by defining our density dependent sequence of Markov processes.  For each $n\in\mathbb N$,
let $Z_{n}(t)$, $t\ge0$,  be an 
irreducible continuous time pure jump Markov process taking values in $\mathbb Z$, with transition
rates given by 
$$
  i\ \to\  i+j \quad \mbox{ at rate }\quad n\lambda_j\Big(\frac{i}{n}\Big),\qquad i \in {\mathbb Z},\ 
  j\in \mathbb Z\setminus \{0\},
$$ 
where the $\lambda_j(\cdot)$ are prescribed functions on $\mathbb R$; we set
$$
   z_{n}(t)\ :=\ n^{-1}Z_{n}(t), \quad t \ge0.
$$
We then define an `average growth rate' of the process $z_{n}$ at $z \in n^{-1}{\mathbb Z}$ by
$$
  F(z)\ :=\ \sum_{j\in \mathbb Z\setminus \{0\}}j{{\lambda}_{j}}(z),
$$ 
and a `quadratic variation' function by $n^{-1}{\sigma}^2(z)$,
where 
$$
  {\sigma}^2(z)\ =\ \sum_{j\in \mathbb Z\setminus \{0\}} j^2  {{\lambda}_{j}}(z),
$$
assumed to be finite for all~$z\in\re$.
\ignore{
For the rate at which the process leaves the state $z$,  we write $n\lambda(z)$, where
$$
   \lambda(z):=\sum_{j\in \mathbb Z\setminus \{0\}} \lambda_j(z).
$$
}

The `law of large numbers' approximation shows that, for large~$n$, the time dependent
development of the process~$z_n$ runs close to the solution of the differential equation system
$\dot z = F(z)$, with the same initial condition, and that there is a
approximately diffusive behaviour on a scale~$n^{-1/2}$ about this path (Kurtz~1970, 71). If~$F$ has a
single zero at a point~$c$, and is such that~$c$ is globally attracting for the
differential equation system, then~$Z_n$ has an 
equilibrium distribution~$\Pi_n$ that is approximately normal,
and puts mass on a scale~$n^{1/2}$ around~$nc$ (Barbour~1980).
The corresponding asymptotic variance is given by $n^{1/2}v_c$ with $v_c := \frac{\sigma^2(c)}{-2F'(c)}$,
provided that $F'(c) < 0$, and the error of the approximation in Kolmogorov distance
is of ideal order~$O(n^{-1/2})$ if only finitely many of the functions~$\l_j$ are non-zero.  

In this paper, we strengthen this result, by proving an accurate approximation to the
equilibrium distribution using another distribution on the integers.  Under 
assumptions similar to those needed for the previous normal approximation, we prove 
that the distance in total variation between the centred
equilibrium distribution $\Pi_{n}-\lfloor nc \rfloor$  and the centred Poisson distribution 
$$
  \widehat{\rm Po}(nv_c)\ :=\ {\rm Po}(nv_c)*\d_{-\lfloor nv_c\rfloor}
$$ 
is of order $O(n^{-1/2})$: here and subsequently, $\d_r$ denotes the point mass on~$r$, 
and~$*$ denotes convolution.  If infinitely many of the~$\l_j$ are allowed to be non-zero,
but satisfy the analogue of a $(2+\a)$'th moment condition, for some $0 < \a \le 1$,
we prove that the error is of order $O(n^{-\a/2})$.

The proof of our approximation runs as follows.
The infinitesimal generator ${\mathcal A}_n$ of~$Z_{n}$, acting on a function $h$,  is given by 
$$
  ({\mathcal A}_n h)(i)\ :=\ 
   \sum_{j\in \mathbb Z\setminus \{0\}}n{{\lambda}_{j}}\Big(\frac{i}{n}\Big)\big[h(i+j)-h(i)\big],
\quad i \in {\mathbb Z}.
$$
In equilibrium, under appropriate assumptions on~$h$, Dynkin's formula implies that
\begin{equation}\label{DF}
  {\mathbb E}({\mathcal A}_n h)(Z_{n})=0.
\end{equation}
The following lemma, whose proof we omit, expresses ${\mathcal A}_n h$ in an alternative form.

\begin{lemma}\label{generator} Suppose that $\sjmo j^2\l_j(z) < \infty$ for all~$z\in\re$.
Then, for any function $h\colon \mathbb Z\to \mathbb R$  with bounded differences,  we have 
\eq\label{basic}
   ({\mathcal A}_nh)(i) 
      \ =\ \frac{n}{2}{\sigma}^2\Big(\frac{i}{n}\Big)\bigtriangledown{g_h}(i)
        +nF\Big(\frac{i}{n}\Big)g_h(i) + E_n(g,i), 
\en
where $\bigtriangledown{f}(i):=f(i)-f(i-1)$ and $g_h(i) := \bigtriangledown{h}(i+1)$
and, for any $i\in\bZ$,
\eqa
   \lefteqn{E_n(g,i)}\non\\ 
  &:=&  -\frac{n}{2}F\Big(\frac{i}{n}\Big)\bigtriangledown{g_h}(i)
      +\sum_{j\geq 2}a_j(g,i) n\lambda_j\Big(\frac{i}{n}\Big) 
      -\sum_{j\geq 2}b_j(g,i) n{{\lambda}_{-j}}\Big(\frac{i}{n}\Big), 
    \label{En-def}
\ena 
with
\eqa
  2a_j(g,i) &:=& -j(j-1)\bigtriangledown g(i)  
     + 2\sum_{k=1}^{j-1} {k \bigtriangledown g(i+j-k)} \label{aj-bnd-1}\\ 
  &=&   2\sum_{k=2}^j \binom{k}{2} \bigtriangledown^2{g_h}(i+j-k+1);\label{aj-bnd-2}\\
  2b_j(g,i) &:=& j(j-1)\bigtriangledown g(i)  
     - { 2\sum_{k=1}^{j-1} k \bigtriangledown g(i-j+k)} \non \\ 
  &=&   2\sum_{k=2}^j \binom{k}{2} \bigtriangledown^2{g_h}(i-j+k). \non
\ena  
\end{lemma}
 
\medskip
Writing \Ref{DF} using the result of Lemma~\ref{generator} leads to the required
approximation, as follows.
In equilibrium, $Z_n/n$ is close to~$c$, as is shown in the next section, and so the main 
part of~\Ref{basic} is close to
$$
   -F'(c)\left\{\frac{n\sigma^2(c)}{-2F'(c)}\bigtriangledown{g_h}(i) - (i-nc)g_h(i)\right\},
$$
because $F(c)=0$.
Here, the term in braces is very close to the Stein operator for the centred Poisson
distribution~$\hP(nv_c)$ with $v_c = \frac{\sigma^2(c)}{-2F'(c)}$, applied to the 
function~$g_h$: see R\"ollin~(2005).  Indeed, for any~$v > 0$ and $B \subset \Z_v$, 
where $\Z_v := \{l \in \mathbb Z,\,l \ge -\lfloor v\rfloor\}$, one can write
\begin{equation}\label{SEt}
  \charI_{B}(l)-\widehat{{\rm Po}}(v)\{B\}\ =\ v\;\bigtriangledown{g}(l+1)-l{g}(l)
             +\langle v \rangle{g}(l) , \quad l \in \Z_v,
\end{equation}
for a function~$g = g_{v,B}$ satisfying
\begin{equation}\label{iSEt}
   \sup_{l\geq -\lfloor v \rfloor}|{g}(l+1)|\ \le\ \min\Bigl\{1,\frac{1}{\sqrt{v}}\Bigr\}; \qquad
   \sup_{l\geq -\lfloor v \rfloor}|\bigtriangledown{g}(l+1)| \ \le\ \frac{1}{v};\qquad
    g(l) \ =\ 0, \quad l \le -\lfloor v \rfloor,
 \end{equation}
where  $\langle x \rangle := x -\lfloor x \rfloor$ denotes the fractional part of~$x$;
note also, from~\Ref{SEt} and~\Ref{iSEt}, that 
\eq\label{iSEt-2}
    \sup_l |l{g}(l)| \le 3.
\en

Replacing~$l$ in~\Ref{SEt} by an integer valued random variable~$W$ then shows that, for
any $B \subset \Z_v$, 
\eqa
   \lefteqn{|\bP[W \in B] - \widehat{{\rm Po}}(v)\{B\}|}\non\\ 
   &&\Le \sup_{g\in\gg_v}
     |\bE\{v\bigtriangledown{g}(W+1)-W{g}(W) +\langle v \rangle{g}(W)\}| 
           + \bP[W < -\lfloor v\rfloor], \label{SEt-2}
\ena
where~$\gg_v$ denotes the set of functions $g\colon\,\bZ\to\re$ satisfying \Ref{iSEt}
and~\Ref{iSEt-2}.
Hence, replacing~$W$ by~$Z_n$ and $v$ by~$nv_c$ in~\Ref{SEt-2}, and comparing the expectation 
with~\Ref{DF} expressed using Lemma~\ref{generator}, the required approximation in total
variation can be deduced; for this part of the argument,
we need in particular to show that, in equilibrium, 
\eq\label{delta-2-bnd}
   |\bE\{\bigtriangledown{g}(Z_n+1)-\bigtriangledown{g}(Z_n)\}|
   \Eq |\bE\{\bigtriangledown^2{g}(Z_n+1)\}| \Eq O(n^{-3/2}),
\en
and also that ${\mathbb E} |E_n(g,Z_n)|
= O(n^{-\a/2})$ for any $g\in\gg_{nv_c}$.  The bound~\Ref{delta-2-bnd} follows from Corollary~\ref{lema2} in 
Section~\ref{one-shift}, and the latter estimate, which also uses~\Ref{delta-2-bnd}, is the substance of 
Section~\ref{main}.

\subsection{Assumptions}
We make the following assumptions on the functions ${\lambda}_{j}$.
The first ensures that the deterministic differential equations have a unique
equilibrium, which is sufficiently strongly attracting. 
\\[0.7ex] 
{\bf \small A1:} There exists a unique $c$ satisfying $F(c)=0$; furthermore, $F'(c) < 0$
    and, for any $\eta>0$, $\mu_{\eta}:= \inf_{|z-c|\geq {\eta}}|F(z)|>0$. 
\\[0.7ex] 
The next assumption controls the global behaviour of the transition functions~$\l_j$.
\\[0.4ex]
{\bf \small A2:} 
  (a)  For each $j\in {\mathbb Z}\setminus \{0,\}$, there exists $c_j\ge0$ such that 
    \begin{equation} \label{lamdaj}
      \lambda_j(z)\leq c_j(1+|z-c|), \qquad z\in \mathbb R,
    \end{equation} 
    where the~$c_j$ are such that, for some $0 < \a \le 1$,
    $$
       \sum_{j \in \mathbb Z \setminus \{0\}}|j|^{2+\alpha}c_{j}<\infty.
    $$
\qquad  (b)  For some $\lambda^0>0$ ,
    $$
      {\lambda}_{1}(z)\ \geq\ 2 \lambda^0,\qquad z\in \mathbb R.
    $$
The moment condition on the~$c_j$ in Assumption~A2\,(a) plays the same r\^ole as
the analogous moment condition in the Lyapounov central limit theorem.  Under
this assumption, the ideal rate of convergence in the usual central limit
approximation is the rate~$O(n^{-\a/2})$ that we establish for our total
variation approximation.
Assumption A2\,(b) is important for establishing the smoothness of the equilibrium
distribution~$\Pi_n$.  If, for instance, all jump sizes were multiples of~$2$, the
approximation that we are concerned with would not be accurate in total variation.
\\[0.7ex] 
We also require some assumptions concerning the local properties of the functions~$\l_j$
near~$c$.\\[0.4ex]
{\bf \small A3:}  
  (a) There exist $\varepsilon>0$ and $0 < \delta \le 1$ and a set $J \subset 
     {\mathbb Z}\setminus \{0\}$ such that
     \eqs
       \inf_{|z-c|\leq \delta}\lambda_j(z)\ \geq\ \varepsilon \lambda_j(c)\ > \ 0,\ \ \ j\in J;\\
       \lambda_j(z)\Eq0 \quad  \mbox{for all}\quad   |z-c|\ \leq\ \delta,\ \ \ j\notin J.
     \ens 
\qquad  (b)\ For each $ j\in J$, ${\lambda}_{j}$  is of class $C^2$ on~$|z-c| \le \d$.  
\\[0.4ex] 
Assumptions A2\,(a) and A3 imply in particular that the series  
$\sum_{j \in \mathbb Z \setminus \{0\}}j\lambda_j(z)$ and 
$\sum_{j \in \mathbb Z \setminus \{0\}}j^2\lambda_j(z)$  are uniformly convergent 
on~$|z-c| \le \d$,  and that their sums, 
$F$ and ${\sigma}^2$ respectively, are continuous there. They also imply that
$$
  \sum_{j\in \mathbb Z\setminus \{0\}}|j|n \lambda_j(i/n)
     \Eq O(|i|),\ \ \ |i|\to \infty,
$$
so that the process $Z_n$ is a.s.\ non-explosive, in view of Hamza and Klebaner~(1995, Corollary~2.1).
\\[0.7ex] 
The remaining assumptions control the derivatives of the functions~$\l_j$ near~$c$.\\[0.4ex]
{\bf \small A4:} For $\delta$ as in A2,  
$$
   L_1 \ :=\ 
    \sup_{j\in J}\frac{\|{{\lambda}_j^{\prime}}\|_{\delta}}{\lambda_j(c)}
      \ <\ \infty,
$$
where   $\|f\|_{\delta}:=\sup_{|z-c|\leq \delta}|f(z)|$. \\[0.7ex] 
This assumption implies in particular, in view of Assumptions A2--A3, that the series 
$\sum_{j \in \mathbb Z \setminus \{0\}}j\lambda^{\prime}_j(z)$ and 
$\sum_{j \in \mathbb Z \setminus \{0\}}j^2\lambda^{\prime}_j(z)$ are uniformly 
convergent on $|z-c|\leq \delta$,
that their sums are 
$F^{\prime}$ and~$(\s^2)^{\prime}$  respectively, and that
$F$ and $\s^2$  are of class $C^1$ on $|z-c|\leq \delta$.  \\[0.7ex]
{\bf \small A5:} For $\delta$ as in A2, 
$$
   L_2\ :=\ 
   \sup_{j\in J} \frac{\|{{\lambda }_j^{\prime \prime}}\|_{\delta}}{|j|\lambda_j(c)}
        \ <\ \infty.
$$
This assumption implies, in view of A2--A3, that the series 
$\sum_{j \in \mathbb Z \setminus \{0\}}j\lambda^{\prime \prime}_j(z)$ 
is uniformly convergent on $|z-c|\leq \delta$, its sum is
$F^{\prime \prime}$, 
and~$F$ is of class $C^2$ on $|z-c|\leq \delta.$ \\[0.7ex]

\ignore{
  In the example, {\bf \small Assumption 1} is satisfied if $d>\sum_{j\geq 1}jb_j,$ with
  $c=a/(d-\sum_{j\geq 1}jb_j).$ {\bf \small Assumption 2} is satisfied with $\lambda^0=a/2$ and
  $c_1=\max\{b_1, a+b_1c\},$ and, for instance, for $\delta=c/2$ and $\varepsilon=1/2,$  for $c_j=b_j
  \max\{1,c\},$ with $j\geq 2,$ and $c_{-1}=d \max\{1,c\}.$ For any $\delta>0,$ we have
  $\mu_{\delta}=\delta(d-\sum_{j\geq 1}jb_j).$  {\bf \small Assumption 3} is satisfied 
  if $\sum_{j\geq 1}|j|^{2+\alpha}b_j< \infty,$ in which case the other assumptions follow immediately.
}

Our arguments make frequent use of the following theorem, which is a restatement in our 
setting of Hamza and Klebaner~(1995, Theorem~3.2), and justifies~\Ref{DF}.

\begin{theorem}\label{HK95} 
Suppose that $Z_{n}$ is non-explosive. Let $h$ be a function satisfying
\eq\label{HK-1}
   (|{\mathcal A}_n|h)(i)\ :=\ 
    \sum_{j\in \mathbb Z\setminus \{0\}}{{\lambda}_{j}}\Big(\frac{i}{n}\Big)|h(i+j)-h(i)|
    \ \leq\ c_{n,h}(1\vee |h(i)|),\;\;\;|i|\rightarrow \infty,
\en 
for some $c_{n,h} < \infty$.  Then,
if $h(Z_n(0))$ is integrable,  so is  $h(Z_n(t))$ for any $t\ge0$; moreover, 
$$
   h(Z_n(t))-h(Z_n(0))-\int_0^t({\mathcal A}_nh)(Z_n(s)))ds
$$ 
is a martingale, and {\it Dynkin's formula} holds:  
\eq\label{HK-2}
  {\mathbb E}[h(Z_{n}(t))-h(Z_n(0))]\ =\ \int_0^t{\mathbb E}({\mathcal A}_n h)(Z_{n}(s))ds.
\en
\end{theorem}

\setcounter{equation}{0} 
\section{Existence of the equilibrium distribution}\label{equilibrium}
In this section, we prove that~$Z_n$ has an equilibrium  distribution which is suitably
concentrated in the neighbourhood of~$nc$.
  
\begin{theorem}\label{exubj} Under Assumptions A1--A4, for all~$n$  large enough, $Z_n$ has an
equilibrium distribution $\Pi_n$, and
\begin{equation} \label{2ineq}
 \begin{split}
 &{\mathbb E}_{\Pi_n}\{|z_{n}-c| \cdot \charI(|z_{n}-c| > \delta)\}\ =\ O(n^{-1})\\
 &{\mathbb E}_{\Pi_n}\{(z_{n}-c)^2 \cdot \charI(|z_{n}-c|\le \delta)\}\ =\ O(n^{-1}),
 \end{split}
 \end{equation}
 for $\delta$ as in Assumption~A3: here, as before, $z_n := n^{-1}Z_n$. 
\end{theorem}

\begin{proof}  The argument is based on suitable choices of Lyapounov functions.
Consider the twice continuously differentiable function $V\colon \mathbb R \to \mathbb R_+$ defined by
$V(z):=|z-c|^{2+\alpha},$ for the $\alpha$ in Assumption~A2\,(a). 
Since $V(c)=0$ and $V(z)>0$ for any $z\neq c$, and because
\eq\label{FVdash}
   F(z)V^{\prime}(z) \Eq -|F(z)|(2+\alpha)|z-c|^{1+\alpha}\ <\ 0\ \ \ \mbox{for any}\ z\neq c,
\en
while $F(c)V^{\prime}(c)=0$, we conclude that $V$ is a Lyapounov function 
guaranteeing the asymptotic stability of the constant solution $c$ of the equation 
$\dot{x}=F(x)$.  We now use it to show the existence of~$\Pi_n$.
 
\begin{lemma} 
\label{lem-3.2}
Under the assumptions of Theorem~\ref{exubj},  the function 
$h_V(i):=V\big(\frac i n\big)=\big|\frac i n -c\big|^{2+\alpha}$  fulfils the
conditions of Theorem~\ref{HK95} with respect to the initial distribution
$\delta_l$, the point mass at~$l$, for any $l\in \mathbb Z$.
\end{lemma}
  
 \begin{proof}
 Checking~\Ref{HK-1}, we use Taylor approximation and Assumption~A2\,(a) to give
\begin{eqnarray}
  (|{\mathcal A}_n|\, {h_V})(i)
& \leq& (2+\alpha)|z-c|^{1+\alpha}\sum_{j\in \mathbb Z \setminus \{0\}}|j|\ c_j(1+|z-c|) \non\\
&& \qquad\mbox{}+\frac{(2+\alpha)(1+\alpha)|z-c|^{\alpha}}{2n}
      \sum_{j\in \mathbb Z \setminus \{0\}} j^2c_j(1+|z-c|) \label{|an|} \\
&& \qquad\mbox{} +\frac{(2+\alpha)(1+\alpha)}{2n^{1+\alpha}}
       \sum_{j\in \mathbb Z \setminus \{0\}}|j|^{2+\alpha}c_j(1+|z-c|),
\end{eqnarray}
where we write $z := i/n$. For $|z-c|<\delta\le 1$, the estimate in~\Ref{|an|} is 
uniformly bounded by
$$
   C_{1n}\ :=\ 2(2+\alpha)\Big\{\sum_j |j| c_j+\frac{(1+\alpha)}{2n}\sum_j j^2
       c_j+\frac{(1+\alpha)}{2n^{1+\alpha}}\sum_j |j|^{2+\alpha} c_j\Big\} \ <\ \infty,
$$
because of Assumption~A2\,(a); for $|z-c|\geq \delta$, we have the bound  
$$
   (|{\mathcal A}_n|\,{h_V})(i)\ \leq\  C_{1n}|z-c|^{2+\alpha} \Eq C_{1n}\, {h_V}(i),
$$
as required.
\end{proof}

The above lemma allows us to apply Dynkin's formula to the function $h_V$. Using Taylor
approximation as for~\eqref{|an|}, but now noting that the first order term
$$
   \sum_{j\in \mathbb Z\setminus \{0\}}\lambda_j(z) j V^{\prime}(z) \Eq F(z)V'(z)
$$
can be evaluated using~\Ref{FVdash}, it follows that
\begin{equation} \label{pti2}
 ({\mathcal A}_n \,h_V)(i)
  \ \leq\ -|F(z)|(2+\alpha)|z-c|^{1+\alpha} + n^{-1}C_2 \Le n^{-1}C_2 
\end{equation} 
on $|z-c|\le\delta$, for 
$$
   C_2 \Eq (2+\alpha) (1+\alpha)\Bigl\{\sum_j j^2 c_j+ \sum_j |j|^{2+\alpha} c_j\Bigr\}
   \ <\ \infty,
$$
where, once again,  $z:= i/n$. 
On $|z-c|> \delta$ and under Assumption A2\,(a),  we have
\begin{eqnarray}  \label{celneg}
 ({\mathcal A}_n \,h_V)(i)  
  &\le& -|F(z)|(2+\alpha)|z-c|^{1+\alpha}\ \non\\
  &&\quad\quad \Big[1-  \frac{(1+\alpha)}{2n|F(z)|\cdot |z-c|}\sum_{j\in \mathbb Z \setminus \{0\}}j^2c_j(1+|z-c|)\non\\
  &&\quad\mbox{}\quad\quad -\frac{(1+\alpha)}{2n^{1+\alpha}|F(z)|\cdot |z-c|^{1+\alpha}}
    \sum_{j\in \mathbb Z \setminus \{0\}}|j|^{2+\alpha}c_j(1+|z-c|)\Big] \non\\  
  &\le& -\frac {\mu_{\delta}(2+\alpha)} 2 |z-c|^{1+\alpha}\ 
     \leq\ - \mu_{\delta} |z-c|^{1+\alpha},  
\end{eqnarray}
 as long as $n$ is large enough that $n\d \ge 1$ and
$$
   \frac{(1+\d)(1+\alpha)}{n\d\mu_{\delta}}\sum_{j\in \mathbb Z \setminus \{0\}}|j|^{2+\alpha}c_j
     \ <\ \frac 1 2.
$$
Dynkin's formula~\Ref{HK-2} then implies, for such $n$, that
\begin{eqnarray*} 
  0&\leq &{\mathbb E}_{i} h_V(Z_{n}(t)) 
     \Eq V(z) + \int_0^t{\mathbb E}_{i} ({\mathcal A}_n \,h_V)(Z_{n}(s))\,ds\\
  &\leq & V(z) + \int_0^t\frac{C_2}{n}{\mathbb P}_{i}(|n^{-1}Z_{n}(s)-c|< \delta)\, ds\\
  &&\mbox{}\quad -\mu_{\delta} \int_0^t{\mathbb E}_{i}\{|n^{-1}Z_{n}(s)-c|^{1+\alpha} \cdot 
     \charI{(|n^{-1}Z_{n}(s)-c|\geq \delta)}\}\,ds,
\end{eqnarray*}
for any $t>0$ and $i \in  \mathbb Z$, where ${\mathbb P}_{i}$ and~${\mathbb E}_{i}$ denote
probability and expectation conditional on $Z_n(0) = i$. It now follows, for any $y\geq \delta$, that 
\begin{eqnarray} \label{totpti2}
  \lefteqn{\frac {\mu_{\delta}\, y^{1+\alpha}} {t} \int_0^t{\mathbb P}_{i}(|{n^{-1}Z_{n}(s)}-c|\geq y)\,ds}\non\\
  &&\Le \frac {\mu_{\delta}}{t} 
     \int_0^t{\mathbb E}_{i}\{|{n^{-1}Z_{n}(s)}-c|^{1+\alpha} \cdot \charI{(|{n^{-1}Z_{n}(s)}-c|\geq y)}\}\,ds\non\\
  &&\Le \frac 1 t V({z})+\frac{C_2}{nt} \int_0^t {\mathbb P}_{i}(|{n^{-1}Z_{n}(s)}-c|< \delta)\,ds,
\end{eqnarray}
and, by letting $t\to \infty$, it follows that
$$
  \limsup_{t\to \infty}\frac 1 t\int_0^t{\mathbb P}_{i}(|{n^{-1}Z_{n}(s)}-c|\geq y)\,ds
    \ \leq\ \frac{C_2}{n\mu_\d\, y^{1+\alpha}}.
$$
This implies that a limiting equilibrium distribution $\Pi_n$ for $Z_n$ exists, 
see for instance Ethier and Kurtz (1986, Theorem~9.3, Chapter~4),
and that, writing $z_n := n^{-1}Z_n$, we have
$$
   {\mathbb P}_{\Pi_n}(|z_{n}-c|\geq y)\ \leq\ 
     \frac{C_2}{n\mu_{\delta}\, y^{1+\alpha}}, 
$$ for any $y\geq \delta$. Furthermore, 
\eqs
{\mathbb E}_{\Pi_n}\{|z_{n}-c| \cdot \charI(|z_{n}-c|\geq \delta)\} &=&
     \int_{\delta}^\infty{\mathbb P}_{\Pi_n}(|z_{n}-c|\geq y)\,dy \\
    &\leq& \int_{\delta}^\infty \frac{C_2}{n\mu_{\delta}\, y^{1+\alpha}}\,dy \Eq O(n^{-1}), 
\ens
proving the first inequality in~\eqref{2ineq}.\\

For the second inequality in~\eqref{2ineq},  we define a  function $\tilde{V}\colon \mathbb R \to
\mathbb R$, which is of class $C^2(\mathbb R)$, is bounded and has uniformly
bounded first and second derivatives on $\mathbb R$,  fulfils the conditions of Theorem~\ref{HK95}, 
and satisfies $F(z)\tilde{V}^{\prime}(z)=-(z-c)^2$ on $|z-c|\leq \delta$. 

In view of the latter property, we begin by letting $v\colon [c-\delta,c+\delta]\to \mathbb R_+$ be 
the function defined by
$$
 v(z)\ :=\ \int_c^z \frac{-(x-c)^2}{F(x)}\,dx,
$$ 
with $v(c)=0$. Note that $v$ is well defined, since  $F^{\prime}(x)<0$ on a  small
enough neighborhood of~$c$, by Assumptions A1 and~A4, and that $v(z)>0$ for 
any $z\neq c$. Furthermore, in view of Assumptions A1 and~A4,  
$$
   v^{\prime }(z) \Eq -\frac{(z-c)^2}{F(z)} \ \ {\rm and}\ \ v^{\prime
    \prime}(z) \Eq \frac{(z-c)^2F^{\prime}(z)-2(z-c)F(z)}{F^2(z)}
$$ 
exist and are continuous on $|z-c|\leq \delta$,
since $|F(z)|>0$ for $z\neq c$, $F(z)\sim F^{\prime}(c)(z-c)$ for $z\to c$, and $F^{\prime}$ is
continuous.  In particular, we have 
\eq\label{Rk-1}
  v^{\prime }(c) \Eq \lim_{z\to c}v^{\prime}(z)=0 \ \ {\rm and}\ \
    v^{\prime \prime}(c) \Eq \lim_{z\to c}v^{\prime \prime}(z) \Eq -\frac 1{F^{\prime}(c)}\ >\ 0.
\en 

Now define the function $\tilde{V}$ to be identical with~$v$ on $|z-c| \le \d$, and 
continued
in $z \le c-\d$ and in $z \ge c+\d$ in such a way that the function is still~$C_2$, 
and takes the same fixed value everywhere
on $|z-c| \ge 2\d$.
\ignore{
as follows:
$$
  \tilde{V}(z)\ :=\  
  \begin{cases} 
    v(c-\delta)-\delta+|z-c| + [v^{\prime}(c-\delta)+1]\sin(z-c+\delta)\\
    \hspace{1.5in}\mbox{}+\frac 1 2 v^{\prime \prime}(c-\delta)\sin^2(z-c+\delta), &{\rm if}\ z<c-\delta; \\
    v(z), &{\rm if}\ |z-c|\leq \delta; \\
    v(c+\delta)-\delta+|z-c|+[v^{\prime}(c+\delta)-1]\sin(z-c-\delta)\\
    \hspace{1.5in}\mbox{}+\frac 1 2 v^{\prime \prime}(c+\delta)\sin^2(z-c-\delta), &{\rm if}\ z>c+\delta. 
  \end{cases}
$$
Note that the function $\tilde{V}$ is of class $C^2(\mathbb R)$, and that 
$$
  |\tilde{V}^{\prime}(z)|\leq C_3 \ \ {\rm and}\ \ |\tilde{V}^{\prime \prime}(z)|\leq C_3,
$$  
for any $z \in \mathbb R$, where
$$
  C_3:=\max\{2+|v^{\prime}(c-\delta)|+3|v^{\prime \prime}(c-\delta)|,\sup_{|z-c|\leq \delta}(|v^{\prime }(z)|+|v^{\prime
 \prime}(z)|),2+|v^{\prime}(c+\delta)|+3|v^{\prime \prime}(c+\delta)|\}.$$
}
Let
\[
   C_3 \ :=\ \max\{\sup_{z\in \mathbb R} \tV(z),\,\sup_{z\in \mathbb R} |\tV'(z)|,\,
            \sup_{z\in \mathbb R} |\tV''(z)|\}.
\]

\begin{lemma}\label{lem-3.3} 
Under the assumptions of Theorem~\ref{exubj}, the function 
$\tilde{h}_V(i):=\tilde{V}\big(\frac i n\big)$ fulfils the
conditions of Theorem~\ref{HK95} with respect to the initial distribution $\Pi_n$.
\end{lemma}
  
 \begin{proof}
 Since $\tilde{h}_V(i)$ is bounded, it follows that $\mathbb E_{\Pi_n}|\tilde{h}_V(Z_n)|<\infty$.
$|{\mathcal A}_n|\,\tilde{h}_V$ is also bounded, since, for $|n^{-1}i-c| \le 4\d$,
by Assumption A2\,(a),
\[
   (|{\mathcal A}_n|\,\tilde{h}_V)(i)
       \Le C_3 \sum_{j\in \mathbb Z\setminus \{0\}} c_j(1+4\d),
\]
while, for $|n^{-1}i-c| > 4\d$,
\eqs
   \lefteqn{(|{\mathcal A}_n|\,\tilde{h}_V)(i)
       \Le C_3 \sum_{j\colon\,|j+i-nc| \le 2n\d} c_j(1+|n^{-1}i-c|)}\\
   &&\Le C_3 \Bigl\{\sum_{j\in \mathbb Z\setminus \{0\}} jc_j\Bigr\} 
        \frac{1+|n^{-1}i-c|}{|i-nc|-2n\d} 
        \Le C_3 \Bigl\{\sum_{j\in \mathbb Z\setminus \{0\}} jc_j\Bigr\}
        \frac{1+4\d}{2n\d}.
\ens
\ignore{
Taylor expansion then yields
\begin{eqnarray*} \
(|{\mathcal A}_n|\,\tilde{h}_V)(i)
   &\leq& C_3\Big(\sum_{j\in \mathbb Z\setminus \{0\}}|j|c_j+\frac{1}{2n}\sum_{j\in \mathbb Z\setminus
       \{0\}}j^2c_j\Big)(1+|n^{-1}i-c|)\\
   &=& O(1\vee |\tilde{h}_V(i)|),
\end{eqnarray*}
by Assumption A2\,(a).
}
\end{proof}

 We now apply Dynkin's formula to $\tilde h_{V}$,  obtaining 
 $$
    0 \Eq \mathbb E_{\Pi_n}\{({\mathcal A}_n\,\tilde h_{V})(Z_n)\}
    \Le \mathbb E_{\Pi_n}\Big\{F(z_n)\tilde{V}^{\prime}(z_n)+
     \sum_{j\in \mathbb Z \setminus \{0\}}\lambda_j(z_n)\frac{j^2}{2n}\, C_3\Big\}.
$$ 
Hence it follows that
\begin{eqnarray*}
  \lefteqn{\mathbb E_{\Pi_n}\{-F(z_n)\tilde{V}^{\prime}(z_n)\cdot\charI(|z_n-c|\le\delta)\}}\\
  &\le& \mathbb E_{\Pi_n}\Big\{F(z_n)\tilde{V}^{\prime}(z_n)\cdot\charI(|z_n-c| > \delta)
     +\sum_{j\in \mathbb Z \setminus \{0\}}\lambda_j(z_n)\frac{j^2}{2n}\, C_3\Big\},
\end{eqnarray*}
whence we obtain
{\allowdisplaybreaks 
\begin{eqnarray*}
  \lefteqn{ \mathbb E_{\Pi_n}\{(z_n-c)^2\cdot \charI(|z_n-c|\le\delta)\}}\\
  &\leq& \mathbb E_{\Pi_n}\{|F(z_n)\tilde{V}^{\prime}(z_n)|\cdot\charI(|z_n-c|> \delta)\}
     + C_3 \,\mathbb E_{\Pi_n}\Big\{\sum_{j\in \mathbb Z \setminus \{0\}}\lambda_j(z_n)\frac{j^2}{2n}\Big\}\\
  &\leq& C_3 \sum_{j\in \mathbb Z\setminus \{0\}}\Big(2|j|+\frac{j^2}{n}\Big)c_j\ 
      \mathbb E_{\Pi_n}\{|z_n-c|\cdot \charI(|z_n-c| > \delta)\}
      + \frac {C_3} {2n} \sup_{|z-c|\le\delta}  \sigma^2(z).
\end{eqnarray*}} 
Using the first inequality in~\eqref{2ineq} and Assumptions A2 and~A3,  we conclude that
$$
  {\mathbb E}_{\Pi_n}\{(z_{n}-c)^2 \cdot \charI(|z_{n}-c|\le \delta)\} \Eq O(n^{-1}),
$$
proving the second inequality in~\eqref{2ineq}. 
\end{proof}

\begin{corollary} \label{prop1} 
Under  Assumptions A1--A4, 
$$
   {\mathbb E}_{\Pi_n}\{|z_{n}-c|\} \Eq  O(n^{-1/2}).
$$
\end{corollary}
 
\begin{proof} 
Using H\"older's inequality, we obtain
\eqs
  \lefteqn{{\mathbb E}\{|z_{n}-c|\}}\\
  &=&{\mathbb E}_{\Pi_n}\{|z_{n}-c| \cdot \charI(|z_{n}-c| > \delta)\}+ 
      {\mathbb E}_{\Pi_n}\{|z_{n}-c|\cdot \charI(|z_{n}-c| \le \delta)\}\\
  &\leq& {\mathbb E}\{|z_{n}-c|\cdot \charI(|z_{n}-c| > \delta)\}+ 
      \sqrt{{\mathbb E}_{\Pi_n}\{(z_{n}-c)^2\cdot \charI(|z_{n}-c| \le \delta)\}}.
\ens
The corollary now follows from Theorem~\ref{exubj}.
\end{proof}

\begin{corollary} \label{prob-n1} 
Under Assumptions A1--A4, for any $0 < \d' \le \d$,
$$
   {\mathbb P}_{\Pi_n}[|z_{n}-c|> \d'] \Eq  O(n^{-1}).
$$
\end{corollary}

\begin{proof} 
It follows from Chebyshev's inequality and Theorem~\ref{exubj} that
$$
   {\mathbb P}_{\Pi_n}[|z_{n}-c|I[|z_n-c| \le \d] > \d'/2] \Le
     4{\mathbb E}_{\Pi_n}\{|z_{n}-c|^2I[|z_n-c| \le \d]\}/(\d')^2 \Eq O(n^{-1}),
$$
and that
$$
   {\mathbb P}_{\Pi_n}[|z_{n}-c| > \d] \Le {\mathbb E}_{\Pi_n}\{|z_{n}-c|I[|z_n-c| > \d]\}/\d
      \Eq O(n^{-1}),
$$
from which the corollary follows.
\end{proof}

\section{The distance between $\Pi_n$ and its unit translation}\label{one-shift}
\setcounter{equation}{0}
\ignore{ 
 If the equilibrium distribution $\Pi_n,$ suitably translated, is indeed $O(1/\sqrt{n})$ close to a
 Poisson distribution with parameter $n\rho,$ say, then its unit translate is correspondingly close to
 ${\rm Po}(n\rho)$  translated by 1. Then, since the total-variation distance between ${\rm Po}(n\rho)$ 
and ${\rm Po}(n\rho)+1$ is of order $O(1/\sqrt{n}),$ the same has to be true of the total-variation 
distance between $\Pi_n$ and its unit translate, where the \emph{distance in total variation} between 
two probability measures $P$ and $Q$ on ${\mathbb Z}$
is defined as:
 $$d_{TV}\{P,Q\}:=\sup_{A \subset \mathbb Z}|P(A)-Q(A)|.$$ However, as illustrated in Barbour and 
Xia~(1999), Barbour and Cekanavicius~(2002) and R\"ollin~(2005), it is extremely useful to be able 
to establish this latter fact in advance, in order to prove the translated Poisson approximation 
heorem, using Stein's method, in the same way that proving a concentration inequality is a useful 
prerequisite for deriving Berry-Essen approximations in the central limit theorem, see
Chen and Shao~(2003). In this section, we establish such a bound.
}
 
A key step in the argument leading to our approximation is to establish that the equilibrium 
distribution~$\Pi_n$ of~$Z_n$ is sufficiently
smooth.  In order to do so, we first need to prove an auxiliary result, showing that, 
if the process~$Z_n$ starts near enough to~$nc$,
then it remains close to~$nc$ with high probability over any finite time interval.
This is the substance of the following lemma.

\begin{lemma}\label{lema3}
Under Assumptions A1--A4, for any $0 < \h \le \d$, there exists a constant $K_{U,\h}<\infty$ such that
$$
   \mathbb P[\sup_{t\in [0,U]}|Z_n(t)-nc|> n\h  \mid Z_n(0) = i]\ \leq\  n^{-1}K_{U,\h},
$$ 
uniformly in $|i - nc| \le n\h e^{-K_1U}/2$, where $K_1 := \|F'\|_{\d}$.
\end{lemma}
 
\begin{proof}
It follows directly from Assumption A2\,(a) that~$h$ defined by $h(j)=j$ satisfies 
condition~\Ref{HK-1}. Fix $Z_n(0) = i$, and define 
\eq\label{tau-def}
   \t_\h \ :=\ \inf\{t\ge0\colon\,|Z_n(t)-nc| > n\h\}.
\en
Then it follows from Theorem~\ref{HK95} that
$$
   {\mathcal M}_n(t)\ :=\ Z_n(t\wedge\t_\h) - i - \int_0^{t\wedge\t_\h} nF(z_n(s))\,ds
$$ 
is a martingale with expectation 0, and with expected quadratic variation no larger than 
\eq\label{QV-2}
   nt \sum_{j \in \bZ\setminus\{0\}}j^2 c_j(1 + \h)
\en
at time~$t$ (see Hamza and Klebaner~(1995, Corollary 3)); here, as earlier, 
$z_n := n^{-1}Z_n$.  Hence we have
\begin{equation*}
 |z_n(t\wedge\t_\h)-c|\ \leq\ \frac 1 n \left\{\sup_{s\in [0,U]}|{\mathcal M}_n(s)|+|i-nc|\right\}
    +\int_0^{t\wedge\t_\h}|F(z_n(s))|\,ds,
\end{equation*}
for any $0\leq t \leq U$, and also, from Assumptions A1--A4, we have 
\begin{equation*}
   |F(z)| 
      \Eq |F(z)-F(c)|\ \leq\ \sup_{|y-c|\leq \d}|F^{\prime}(y)|\, |z-c|.  
\end{equation*}
Hence it follows that
\begin{equation*}
   \int_0^{t\wedge\t_\h}|F(z_n(s))|\,ds\ \leq \ K_1 \int_0^{t\wedge\t_\h}|z_n(s)-c|\,ds.
\end{equation*}  
Gronwall's inequality now implies that
\begin{equation*} 
    |z_n({t\wedge\t_\h})-c|\ \leq\  
       n^{-1}\left\{\sup_{s\in [0,U]}|{\mathcal M}_n(s)|+|i-nc|\right\}e^{K_1 t},
\end{equation*}
for any $0\leq t \leq U$, and so, for $|i - nc| \le n\h e^{-K_1U}/2$,
\eq\label{wedge-bnd}
    \sup_{t\in [0,U]}|z_n(t\wedge\t_\h)-c| 
        \Le \h/2 + n^{-1}\sup_{s\in [0,U]}|{\mathcal M}_n(s)|e^{K_1 U}.
\en
We have thus shown that
\begin{equation}\label{zc}
   \mathbb P[\sup_{t\in [0,U]}|z_n(t)-c|> \h \mid Z_n(0)=i\} \Le
        \mathbb P[\sup_{s\in [0,U]}|{\mathcal M}_n(s)| > ne^{-K_1 U}\h/2 \mid Z_n(0)=i].
\end{equation}
But by Kolomogorov's inequality, from~\Ref{QV-2}, we have
\begin{equation}\label{kolm}
    \mathbb P[\sup_{s\in [0,U]}|{\mathcal M}_n(s)| > ne^{-K_1 U}\h/2 \mid Z_n(0)=i] 
        \Le 4 n^{-1}\h^{-2}e^{2K_1U}  U \sum_{j \in \bZ\setminus\{0\}}j^2 c_j(1 + \h),
\end{equation}
completing the proof.
\end{proof} 

We can now prove the main theorem of this section.

\begin{theorem} \label{thdtvpi} 
Under Assumptions A1--A4, there exists a constant $K>0$ such that
$$
   d_{TV}\{\Pi_{n},\Pi_{n}*\delta_1\}\leq K n^{-1/2},
$$ 
where $\Pi_{n}*\delta_1$ denotes the equilibrium distribution $\Pi_{n}$ of $Z_n$, translated by 1.
\end{theorem}

\begin{proof} 
Because we have little {\it a priori\/} information 
about $\Pi_n$, we fix any~$U>0$, and use the stationarity of~$\Pi_n$ to give the inequality
\begin{eqnarray}
  \lefteqn{d_{TV}\{\Pi_{n},\Pi_{n}*\delta_1\}}\non\\
  &\leq& \sum_{i \in {\mathbb Z}} \Pi_{n}(i)\, d_{TV}\{{\mathcal L}(Z_{n}(U) \mid 
    Z_{n}(0)=i),{\mathcal L}(Z_{n}(U)+1 \mid Z_{n}(0)=i)\}, \label{disttv}
\end{eqnarray}
By Corollary~\ref{prob-n1}, we thus have, for any $\d' \le \d$,
\eq\label{bound-1}
  d_{TV}\{\Pi_{n},\Pi_{n}*\delta_1\} \Le D_{1n}(\d') + O(n^{-1}),
\en
where
$$
   D_{1n}(\d') \ :=\ \sum_{i\colon\,|i-nc| \le \d'}\Pi_{n}(i)\, d_{TV}\{{\mathcal L}(Z_{n}(U) \mid 
     Z_{n}(0)=i),{\mathcal L}(Z_{n}(U)+1 \mid Z_{n}(0)=i)\}.
$$
This alters our problem to one of finding a bound of similar form, but now
involving the transition probabilities of the
chain~$Z_n$ over a finite time~$U$, and started in a fixed state~$i$ which is relatively 
close to~$nc$.

We now use the fact that the upward jumps of length~$1$ occur at least as fast as a Poisson
process of rate~$\l^0$, something that will be used to derive the smoothness that we require.
We realize the chain~$Z_n$ with $Z_n(0) = i$ in the form $N_n + X_n$, for the bivariate chain 
$(N_n,X_n)$ having transition rates
$$
\begin{array}{ll}
 (l,m) \rightarrow (l+1,m) & \mbox{at rate}\ n\lambda^0 \\
 (l,m) \rightarrow (l,m+1) 
       & \mbox{at rate}\ n\Bigl[{{\lambda}_{1}}\Bigl(\frac{l+m}{n}\Bigr)-\lambda^0\Bigr]\\
 (l,m) \rightarrow (l,m+j) & \mbox{at rate}\ n{{\lambda}_{j}}\Bigl(\frac{l+m}{n}\Bigr),
     \ \mbox{for any}\; j\in \mathbb Z,\, j\neq 0,1,
\end{array}
$$
and starting at $(0,i)$. This allows us to deduce that
\eqa
  \lefteqn{d_{TV}\{{\mathcal L}(Z_{n}(U) \mid Z_{n}(0)=i),{\mathcal L}(Z_{n}(U)+1 \mid Z_{n}(0)=i)\}}\non\\
  &=& \frac{1}{2}\sum_{k \in {\mathbb Z}}|{\mathbb P}(Z_{n}(U)=k \mid Z_{n}(0)=i)
     -{\mathbb P}(Z_{n}(U)=k-1 \mid Z_{n}(0)=i)| \non\\
  &=& \frac{1}{2}\sum_{k \in {\mathbb Z}}\left|\sum_{l\geq 0}{\mathbb P}(N_{n}(U)=l)
    {\mathbb P}(X_{n}(U)=k-l \mid N_{n}(U)=l, X_{n}(0)=i) \right. \non\\
  &&\ \mbox{}\left. - \sum_{l\geq 1}{\mathbb P}(N_{n}(U)=l-1)
    {\mathbb P}(X_{n}(U)=k-l \mid N_{n}(U)=l-1, X_{n}(0)=i) \right| \non\\
  &\le&  \frac{1}{2}\sum_{k \in {\mathbb Z}}\sum_{l \geq 0}
       |{\mathbb P}(N_{n}(U)=l)-{\mathbb P}(N_{n}(U)=l-1)|f^{U}_{l,i}(k-l) \non\\
  && \ \mbox{}+\frac{1}{2}\sum_{k \in {\mathbb Z}}\sum_{l \geq 1}{\mathbb P}(N_{n}(U)=l-1)
      |f^{U}_{l,i}(k-l)-f^U_{l-1,i}(k-l)|, \label{Poisson-part}
\end{eqnarray}
where
\begin{equation}\label{dens1}
    f^{U}_{l,i}(m):={\mathbb P}(X_{n}(U)=m \mid N_{n}(U)=l, X_{n}(0)=i). 
\end{equation}
Since, from Barbour, Holst and Janson~(1992, Theorem~1.C),
\begin{equation}\label{BHJ}
  \sum_{l \geq 0}|{\mathbb P}(N_{n}(U)=l)-{\mathbb P}(N_{n}(U)=l-1)|
   \ \leq\ \frac{1}{\sqrt{n\lambda^0U}} \ =\ O\Big(\frac{1}{\sqrt{n}}\Big),
\end{equation} 
the first term in~\Ref{Poisson-part} is bounded by $1/\{\sqrt{n\lambda^0U}\}$,
yielding a contribution of the same size to~$D_{1n}(\d')$ in~\Ref{bound-1}, and
it remains only to control the differences between the conditional probabilities
$f^{U}_{l,i}(m)$ and~$f^{U}_{l-1,i}(m)$.

To make the comparison between $f^{U}_{l,i}(m)$ and~$f^{U}_{l-1,i}(m)$, we first
condition on the whole Poisson paths of~$N_n$ leading to the events $\{N_{n}(U)=l\}$
and $\{N_{n}(U)=l-1\}$, respectively, chosen to be suitably matched; we write
\eqa
  f^{U}_{l,i}(m) &=&  \frac 1{U^l} \int_{[0,U]^l}ds_1\,\ldots\,ds_{l-1}\,ds^* \non\\ 
   &&\hspace{0.21in}   {\mathbb P}(X_{n}(U)=m \mid N_n[0,U]=\n^l(\cdot\;;s_1,\ldots,s_{l-1}, s^*), X_{n}(0)=i);\non\\
 f^U_{l-1,i}(m) &=& \frac 1 {U^l}\int_{[0,U]^{l}}ds_1...ds_{l-1}ds^*  \non \\
 && \quad  {\mathbb P}(X_{n}(U)=m \mid N_n[0,U]=\n^{l-1}(\cdot\;;s_1,\ldots,s_{l-1}), X_{n}(0)=i) ,
   \label{dens2}
\ena
where
$$
   \n^r(u;t_1, \ldots, t_{r})\ :=\ \sum_{i=1}^{r} \charI_{[0,u]}(t_i),
$$
and $Y[0,u]$ is used to denote $(Y(s),\,0\le s\le u)$.  Fixing $s_1,s_2,\ldots,s_{l-1}$, 
let $\bPstar$ denote the distribution of~$X_n$ conditional on 
$N_n[0,U] = \n^l(\cdot\;;s_1,\ldots,s_{l-1}, s^*)$ and  $X_{n}(0)=i$, and let~$\bP_i$
denote that conditional on 
$N_n[0,U] = \n^{l-1}(\cdot\;;s_1,\ldots,s_{l-1})$ and  $X_{n}(0)=i$; let
$\rstar(u,x)$ denote the Radon--Nikodym derivative $d\bPstar/d\bP_i$  evaulated at the
path $x[0,u]$. Then
$$
    \bPstar[X_n(U)=m]  \Eq \int_{\{x[0,U]\colon x(U)=m\}} \rstar(U,x)\,d\bP_i(x[0,U]),
$$
and hence
\eq\label{pm-diff}
  \bPstar[X_n(U)=m] - \bP_i[X_n(U)=m] \Eq \int \charI_{\{m\}}(x(U))\{\rstar(U,x)-1\}\,d\bP_i(x[0,U]).
\en
Thus
\eqa
  \lefteqn{\sum_{m\in\bZ}|f^U_{l,i}(m) - f^U_{l-1,i}(m)|}\non\\
   &\le& \frac 1{U^l} \int_{[0,U]^l}ds_1\,\ldots\,ds_{l-1}\,ds^* \sum_{m\in\bZ}
      \bE_i\left\{\charI_{\{m\}}(X_n(U))|\rstar(U,X_n)-1|\right\} \non\\
   &\le&  \frac 2{U^l} \int_{[0,U]^l}ds_1\,\ldots\,ds_{l-1}\,ds^*\,
      \bE_i\left\{[1 - \rstar(U,X_n)]_+\right\}. \label{RN-bound}
\ena

To evaluate the expectation, note that $\rstar(u,X_n)$, $u\ge0$, is a $\bP_i$-martingale
with expectation~$1$.
Now, if the path $x[0,U]$ has~$r$ jumps at times $t_1 < \cdots < t_r$, writing 
$$
  y(v) \ :=\ x(v) + \n^{l-1}(v\;;s_1,\ldots,s_{l-1}),\quad y_k \ :=\ y(t_k),\quad
  j_k \ :=\ y_k - y_{k-1},
$$
we have
$$
  \rstar(u,x) \Eq 
  \begin{cases}
      1  
             &\mbox{if}\quad u < s^*; \\
      \exp\left( -n\int_{s^*}^u \{\hl(y({v})+n^{-1}) - \hl(y({v}))\}\,{dv}\right)\\
       \qquad\prod_{\{k\colon s^* \le t_k \le u\}}
        \left\{\hl_{j_k}(y_{k-1} + n^{-1}) / \hl_{j_k}(y_{k-1})\right\}  
            &\mbox{if}\quad u \ge s^*,
  \end{cases} 
$$
where $\hl_j(\cdot) = \l_j(\cdot)$ if $j\ne1$ and $\hl_1(\cdot) = \l_1(\cdot) - \l^0$,
and where $\hl(\cdot) := \sjmo \hl_j(\cdot)$.
Thus, in particular, $\rstar(u,x)$ is
absolutely continuous except for jumps at the times $t_k$.  Then also, from
Assumptions A3\,(a) and~A4,
$$
   \left|\frac{\l_j(y+n^{-1})}{\l_j(y)} - 1 \right| \Le \frac{\|\l'_j\|_\d}{n\e\l_j(c)}
      \Le |j|L_1/\{n\e\},
$$
uniformly in $|y-c| \le \d$, for each $j\in J$.  Hence it follows that, if we define 
the stopping times
\eqa
   \t_\d &:=& \inf\{u\ge0: |X_n(u) + \n^{l-1}(u\;;s_1,\ldots,s_{l-1}) - nc| > n\d\};\non\\
   \f  &:=&  \inf\{u\ge0: \rstar(u,X_n) \ge 2\}, \label{stopping}
\ena
then the expected quadratic variation of the martingale~$\rstar(u,X_n)$ up to the time
$\min\{U,\t_\d,\f\}$ is at most 
\eq\label{QV}
   4U\sum_{j \in \bZ\setminus\{0\}} \left(\frac{|j|L_1}{n\e}\right)^2\,nc_j(1+\d) 
        \ =:\ n^{-1}K(\d,\e) U,
\en
where $K(\d,\e) < \infty$ by Assumption~A2\,(a).

Clearly, from~\Ref{QV} and from Kolmogorov's inequality, 
$$
   \bP_i[\f < \min\{U,\t_\d\}] \Le K(\d,\e)U/n.
$$
Hence, again from~\Ref{QV},
$$
  \bE_i\left\{[1 - \rstar(U,X_n)]_+\right\} \Le n^{-1/2}\sqrt{K(\d,\e) U} + n^{-1}K(\d,\e)U
   + \bP_i[\t_\d < U].
$$
Substituting this into~\Ref{RN-bound}, it follows that
\eqs
   \lefteqn{\sum_{l \geq 1}{\mathbb P}(N_{n}(U)=l-1)\sum_{m\in\bZ}|f^U_{l,i}(m) - f^U_{l-1,i}(m)|}\\
    && \Le 2\Bigl\{n^{-1/2}\sqrt{K(\d,\e) U} + n^{-1}K(\d,\e)U \\
    &&\qquad\qquad\mbox{}   + \bP[\sup_{0 \le u \le U}|Z_n(u) - nc| > n\d \mid Z_n(0) = i]\Bigr\}.
\ens
But now, for all~$i$ such that $|i - nc| \le n\d' = n\d e^{-K_1U}/2$, the latter probability
is of order $O(n^{-1})$, by Lemma~\ref{lema3}, and hence the final term in~\Ref{Poisson-part}
is also of order $O(n^{-1/2})$, as required.
\end{proof}

\ignore{
 We now need to compare ${\mathcal L}(Z_{n}(U) \mid Z_{n}(0)=z)$ with ${\mathcal L}(Z_{n}(U)+1 \mid
 Z_{n}(0)=z)$ for any given and fixed $U,$ where $Z_n$ is the Markov process with generator ${\mathcal
 A}_n.$ As in R\"ollin~(2005), a good way of doing this is to find a random variable $N$ embedded in
 $Z_n(U)$ for which ${\mathcal L}(N)$ and ${\mathcal L}(N+1)$ are close enough, and to exploit this. Here,
 we use a Poisson process with jumps of size $+1$ at rate $n\lambda^0$ to provide our $N.$ We thus need to
 split $Z_n$ into a sum of two processes, one of which is this Poisson process. The construction of 
the appropriate bivariate process $(X_n, N_n)$  is done at {\small \bf Step 2}. 
Using the newly defined processes, we are able to bound $d_{TV}\{\Pi_{n},\Pi_{n}*\delta_1\}$
by the sum of two terms, see {\small \bf Step 3}. The first one of these we can further bound by the
total-variation distance between a Poisson distribution and its unit translation. This distance, by
 Barbour, Holst and Janson~(1992, Theorem~1.C),  can be bounded by $O(1/\sqrt{n}),$  
see {\small \bf Step 4}. 
The last step, {\small \bf Step 5}, is concerned with finding a bound of the same size on the second
term in the sum from {\small \bf Step 3}. Here, the essence of the argument is to show that the 
conditional distribution of the paths of $X_n$ on $[0,U],$ given the path of the Poisson 
process $\{N_n(t), \ 0\leq t\leq U\},$ changes only little in total variation if an extra jump 
is added to $N_n:$ that is to say, ${\mathcal L}(X_{n}(U) \mid N_n(t)=n(t), 0\leq t\leq U; X_{n}(0)=z)$ 
and ${\mathcal L}(X_{n}(U) \mid N_n(t)=n(t)+\charI_{[s^*,\infty)}(t), 0\leq t\leq U; X_{n}(0)=z),$ 
where $s^*$ denotes the time of the extra jump, are $O(1/\sqrt{n})$ close. The detailed argument 
requires several sub-steps, that we shall explain in due time.\\
}

As a consequence of this theorem, we have the following corollary.

 \begin{corollary} \label{lema2} Under Assumptions A1--A4, for any bounded function~$f$,
$$
   {\mathbb E}_{\Pi_n}\{\bigtriangledown f(Z_n)\} \Eq O\Big(\frac{1}{\sqrt{n}}\,\|f\|\Big).
$$
\end{corollary}

\begin{proof}  Immediate, because
$$
   |{\mathbb E}_{\Pi_n}\{\bigtriangledown f(Z_n)\}| \Le 2\|f\|\,d_{TV}(\Pi_n,\Pi_n*\d_1).
$$
\end{proof}

 \setcounter{equation}{0}  
 \section{Translated Poisson approximation to the equilibrium distribution}
\label{main}
 
We are now able to prove our main theorem.  The centred equilibrium distribution of~$Z_n$
is $\widehat{\Pi}_n:=\Pi_n*\d_{-\lfloor nc \rfloor}$, and we approximate it by a centred 
Poisson distribution with similar variance.

\begin{theorem}\label{main-thm}
Under Assumptions A1--A5,
$$
  d_{TV}(\widehat{\rm Po}(nv_c),\widehat{\Pi}_n) \Eq O(n^{-\a/2}), 
$$ 
where $v_c := \s^2(c)/\{-2F'(c)\}$.
\end{theorem}

\ignore{
For the example in the Preliminaries Section, we have $c=\frac{a}{d-\sum_{j\geq 1}jb_j},$
$-2F^{\prime}(c)=2(d-\sum_{j\geq 1}jb_j)$ and $\sigma^2(c)=a+\frac{a}{d-\sum_{j\geq 1}jb_j}(d+\sum_{j\geq 1}j^2b_j),$ so that the shifted equilibrium distribution $$\Pi_n-\Big\lfloor \frac{na}{d-\sum_{j\geq 1}jb_j}\Big\rfloor$$
 can be  approximated in total variation by the shifted Poisson distribution 
 $$\widehat{\rm Po}\Big(\frac{na(2d+\sum_{j\geq 1}j(j-1)b_j)}{2(d-\sum_{j\geq 1}jb_j)^2}\Big),$$
 to order $O(1/\sqrt{n}).$
  Note that if $b_j=0,$ for any $j\geq 1,$ than the process becomes a simple death with immigration process, whose equilibrium distribution is precisely the Poisson distribution 
  ${\rm Po}\big(\frac{na}{d}\big)={\rm Po}(nc).$ In this case, the approximation is in fact exact.
}

\begin{proof}
We follow the recipe outlined in Section~\ref{prelims}.  
From~\Ref{SEt-2}, we principally need to show that
$$
  \sup_{g\in\gg_{v}}
    |\bE\{{v}\;\bigtriangledown{g}(W+1)-W{g}(W) +\langle v \rangle{g}(W)\}| \Eq O(n^{-\a/2}),
$$
for $W := Z_n - \intnc$, $v := nv_c$ and~$\bE := \bE_{\Pi_n}$.
So, for any $g\in\gg_{nv_c}$,  write $\tg(i) := g(i - \intnc)$, and set 
$$
  h \ :=\ h_{n,g}(i)\ :=\ 
  \begin{cases}
     0,  &\mbox{if}\quad i \le \lfloor nc \rfloor - \lfloor {nv_c} \rfloor;\\ 
     \sum_{l =  \intnc - \lfloor {nv_c} \rfloor}^{i  - 1}\tg(l)
         &\mbox{if}\quad i > \lfloor nc \rfloor - \lfloor {nv_c} \rfloor.
  \end{cases}
$$
Note that, for $j\ge1$, by Assumption A2\,(a),
\eqs
  n\l_j(i/n)|h(i+j)-h(i)| &\le& njc_j\|\tg\| + c_j|i-\lfloor nc \rfloor|
     \sum_{k=1}^j |g(i+j-k-\lfloor nc \rfloor)|\\
  &\le& njc_j\|g\| + jc_j \sup_l|lg(l)| + c_j\sum_{k=1}^j |j-k|\|g\|,
\ens
and that a similar bound, with $|j|$ replacing~$j$, is valid for $j\le -1$. 
From the definition of~$\gg_{nv_c}$ in \Ref{SEt} and~\Ref{iSEt} and from Assumption~A2\,(a), 
it thus follows
that $(|{\mathcal A}_n|h_{n,g})$ is a bounded function, and hence that
the function~$h_{n,g}$ satisfies condition~\Ref{HK-1}; furthermore, since 
$|h_{n,g}(i)| \le |i - \lfloor nc \rfloor + \lfloor {nv_c} \rfloor|$, in view of~\Ref{iSEt},
$h_{n,g}$ is integrable with respect to~$\Pi_n$, because of Theorem~\ref{exubj}.
Hence it satisfies the conditions of Theorem~\ref{HK95}, from which we deduce, as
in~\Ref{DF}, that
$$
   \bE_{\Pi_n}({\mathcal A}_n h_{n,g})(Z_n) \Eq 0.
$$

Applying Lemma~\ref{generator}, since~$h_{n,g}$ has bounded differences in view of~\Ref{iSEt},
it follows that
\eqa
   0 &=& \bE_{\Pi_n}\left\{\frac{n}{2}{\sigma}^2\Big(\frac{Z_n}{n}\Big)\bigtriangledown{\tg}(Z_n)
        +nF\Big(\frac{Z_n}{n}\Big)\tg(Z_n) + E_n(\tg,Z_n)\right\}\non \\
     &=& -F'(c)\bE_{\Pi_n}\left\{ nv_c\bigtriangledown{\tg}(Z_n) - 
          (Z_n-\lfloor nc\rfloor)\tg(Z_n) 
             +\langle nv_c \rangle{\tg}(Z_n)\right\}\non \\
     &&\quad\mbox{} + \bE_{\Pi_n}\{ E'_n(\tg,Z_n) + E_n(\tg,Z_n)\}, \label{DE-1}
\ena
where $E_n$ is as defined in~\Ref{En-def}, and
\eqs
    E'_n(g,i) &:=& \frac{n}{2}({\sigma}^2(i/n) - \s^2(c))\bigtriangledown{g}(i)\\
    &&\quad\mbox{}  + \{n(F(i/n) - F(c)) - F'(c)(i-\lfloor nc\rfloor)\}g(i) 
        + F'(c) \langle nv_c \rangle{g}(i).
\ens
The terms involving $E'_n(\tg,i)$ can be bounded, using~\Ref{iSEt}, as follows.  First,
using Assumptions A2\,(a) and~A4,
\eqa
   \lefteqn{ \frac{n}{2}|{\sigma}^2(i/n) - \s^2(c)|\,|\bigtriangledown{\tg}(i)|}\non\\ 
        &&\Le\frac{1}{2nv_c}\|(\sigma^2)^{\prime}\|_{\delta} |i-nc|I[|i-nc| \le n\d] \non\\
     &&\mbox{}\qquad + \frac{1}{2v_c}\Bigl(\sjmo j^2c_j(1 + |i/n-c|) + \s^2(c)\Bigr)I[|i-nc| > n\d];
     \label{e-dash-1}
\ena
and then, under Assumptions A2\,(a) and~A5,
\eqa
    \lefteqn{|n(F(i/n) - F(c)) - F'(c)(i-\lfloor nc\rfloor) + F'(c) \langle nv_c \rangle|
                      \,|\tg(i)|}\non\\
    &&\Eq n|F(i/n) - F(c) - (i/n-c)F'(c)|\,|\tg(i)| \non\\
    &&\Le \Bigl(\frac n2 (i/n-c)^2 I[|i/n-c| \le \d] \sup_{|z-c|\le\d}|F''(z)| \non\\
    &&\mbox{}\qquad  + n\Bigl\{(1+|i/n-c|)\sjmo |j|c_j + F'(c)|i/n-c|\Bigr\}I[|i-nc| > \d]\Bigr)\frac1{\sqrt{nv_c}}.
    \label{e-dash-2}
\ena
The contribution to~\Ref{DE-1} from $\bE_{\Pi_n}\{ E'_n(\tg,Z_n)\}$ is thus of order
\eqa
   \lefteqn{\bE_{\Pi_n}\{|z_n-c| + (1 + |z_n-c|)I[|z_n-c| > \d] + |z_n-c|^2 I[|z_n-c| \le \d]\}}\non\\
    &&\Eq O(n^{-1/2}), \hspace{3.5in}\label{Order-1}
\ena
by Theorem~\ref{exubj} and Corollaries \ref{prop1} and~\ref{prob-n1}.
The first term in $E_n(\tg,i)$ is also bounded in similar fashion:
from Assumptions A1, A2\,(a) and~A4,
\eqa
   \lefteqn{\frac{n}{2}|F(i/n)|\,|\bigtriangledown{\tg}(i)|} \non\\
      &&\Le \frac1{2nv_c} \{\|F'\|_\d|i-nc| +  \sjmo c_j|j|(1 + |i-nc|)I[|i-nc| > \d]\}.
   \label{e-1}          
\ena
giving a contribution to $\bE_{\Pi_n}\{ E_n(\tg,Z_n)\}$ of the same order.
The remaining terms, involving $\bigtriangledown^2{\tg}$, need to be treated more
carefully. 

 We examine the first of them in detail, with the treatment of the second being
entirely similar.  First, if either $|i/n-c| > \d$ or $j > \sqrt n$, it is enough to use
the expression in~\Ref{aj-bnd-1} to give
\eqa
   |a_j(\tg,i)|
     &\le& j(j-1)\|\bigtriangledown{\tg}\| \Le j(j-1)/(nv_c). \label{aj-part1}
\ena
For $|i/n-c| > \d$, by Assumption~A2\,(a), this yields the estimate
\eqa
   \lefteqn{\left|\sum_{j\ge2}a_j(\tg,i) n\l_j(i/n)\right|I[|i-nc| > \d]}\non\\
   &\le& \sum_{j\ge2} \frac{j(j-1)c_j}{v_c}(1 + |i/n-c|)I[|i-nc| > \d], \label{e-2}
\ena
with corresponding contribution to $\bE_{\Pi_n}\{ E_n(\tg,Z_n)\}$ being of order $O(n^{-1})$,
by Theorem~\ref{exubj} and Corollary~\ref{prob-n1}.
Then, for $j > \sqrt n$ and $|i/n-c| \le \d$, \Ref{aj-part1} yields
\eqa
   \lefteqn{\left|\sum_{j>\sqrt n} a_j(\tg,i) n\l_j(i/n)\right|}\non\\
   &\le& \sum_{j>\sqrt n} \frac{j(j-1)c_j}{v_c}(1 + \d) 
                  \Le \sum_{j\ge1}j^{2+\a}c_j n^{-\a/2}(1+\d)/v_c,
     \label{e-3}
\ena
making a contribution of order $O(n^{-\a/2})$ to $\bE_{\Pi_n}\{ E_n(\tg,Z_n)\}$,
again using Assumption~A2\,(a).  
In the remaining case, in which $j \le \sqrt n$ and $|i/n-c| \le \d$, we use
\Ref{aj-bnd-2}, observing first that
\eqa
   \lefteqn{n\bigtriangledown^2{\tg}(i+j-k+1) \l_j(i/n)} \non\\
   &&=\ n\bigtriangledown^2{\tg}(i+j-k+1) \l_j(c) +
        n\bigtriangledown^2{\tg}(i+j-k+1) (\l_j(i/n) - \l_j(c)), \label{e-4}
\ena
the latter expression being bounded by 
\eqa
    |n\bigtriangledown^2{\tg}(i+j-k+1) (\l_j(i/n) - \l_j(c))|
     &\le&  \frac2{v_c}\|\l_j'\|_\d \,|i/n-c|. \label{e-5}
\ena
The corresponding contribution to $\bE_{\Pi_n}\{ E_n(\tg,Z_n)\}$ is thus at most
\eqa
    &&\sum_{j=2}^{\lfloor\sqrt n\rfloor} (j^3/6)
      \{\l_j(c) n\sup_l|\bE_{\Pi_n}\bigtriangledown^2{\tg}(Z_n+l)|
        + 2 v_c^{-1}\|\l_j'\|_\d \,\bE_{\Pi_n}|z_n-c| \}\non\\
    &&\quad\Le n^{(1-\a)/2} \sum_{j\ge2} j^{2+\a}c_j
       \{n \sup_l|\bE_{\Pi_n}\bigtriangledown^2{\tg}(Z_n+l)| 
             + L_1 2 v_c^{-1}\bE_{\Pi_n}|z_n-c| \}\non\\
    &&\quad\Eq {n^{(1-\alpha)/2}\; O\bigl(n\cdot n^{-3/2}+n^{-1/2} \bigr) \Eq O(n^{-\alpha/2})} , \label{e-6}
\ena
where we have used Assumptions A2\,(a) and~A4, and then Corollaries \ref{prop1} and~\ref{lema2},
and finally~\Ref{iSEt}.

Combining the bounds, and substituting them into~\Ref{DE-1}, it follows that
$$
  |\bE_{\Pi_n}\left\{ nv_c\bigtriangledown{g}(Z_n-\intnc) - 
       (Z_n-\lfloor nc\rfloor)g(Z_n-\intnc) 
           +\langle nv_c \rangle{g}(Z_n-\intnc)\right\}|
    \Eq O(n^{-\a/2}),
$$
uniformly in $g\in\gg_{nv_c}$.  Again from Corollary~\ref{lema2}, we also have
$$
   |nv_c \bE_{\Pi_n}\left\{ \bigtriangledown{g}(Z_n-\intnc) - \bigtriangledown{g}(Z_n-\intnc+1)
     \right\}|  \Eq O(n^{-1/2}),
$$
for any $g\in\gg_{nv_c}$. It thus follows from~\Ref{SEt-2} that
$$
   d_{TV}(\widehat{\rm Po}(nv_c),\widehat{\Pi}_n) \Eq 
        O\left(n^{-\a/2} + \bP_{\Pi_n}[Z_n - nc < -\lfloor nv_c\rfloor]\right), 
$$ 
and the latter probability is of order~$O(n^{-1})$ by Corollary~\ref{prob-n1}.
This completes the proof.
\end{proof} 

\medskip
\nin{\bf Example.}\ Consider an immigration birth and death process~$Z$, with births occurring in groups
of more than one individual at a time.  The process has transition rates as in Section~\ref{prelims},
with
$$
  \lambda_{-1}(z):=dz, \ \ \ \lambda_{1}(z):=a+bq_1z \ \ \mbox{and} \ \ \lambda_j(z):=bq_jz, \ j\geq 2,
$$ 
while $\lambda_j(z):=0$, $j<-1$. Here, $b$ denotes the rate at which birth events occur,
and $a>0$ represents the immigration rate.
The quantity~$q_j$ denotes the probability that~$j$ offspring are born at a birth event, so
that $\sum_{j\ge1} q_j = 1$; we write $m_r := \sum_{j\ge1} j^rq_j$ for the $r$'th moment of
this distribution. 
Then 
$$
  F(z)=a+z(bm_1-d), \quad\mbox{and}\quad \sigma^2(z)=a+z(bm_2+d). 
$$
Assumption~A1 is satisfied if $d>bm_1$, with
$c=a/(d-bm_1)$ and $F'(c) = -(d-bm_1)$.  Assumption~A2\,(a) is satisfied with 
$c_j = bq_j \max\{1,c\}$, $j\ge2$, $c_1=\max\{bq_1, a+bq_1c\}$, and $c_{-1}=d \max\{1,c\}$,
provided that $m_{2+\a} < \infty$ for some $0 < \a \le 1$; for Assumption~A2\,(b), simply
take $\lambda^0=a/2$.  The other assumptions are immediate.

The quantity $v_c$ appearing in Theorem~\ref{main-thm} then comes out to be
\[
    v_c \ :=\ \frac{a(2d+b(m_2-m_1))}{2(d-bm_1)^2}\,,
\]
and the approximation to the equilibrium distribution of $Z_n - \lfloor nc \rfloor$
is the centred Poisson distribution $\widehat{\rm Po}(nv_c)$, accurate in total
variation to order $O\bigl(n^{-\a/2}\bigr)$.
Note that, if $b=0$, then the process becomes a simple immigration death 
process, whose equilibrium distribution is precisely the Poisson distribution 
${\rm Po}\big(na/d\big)={\rm Po}(nc)$. In this special case, the approximation is in fact exact.

\ignore{
\section*{Acknowledgment}
The author would like to thank her Ph.D. advisor  A.~D.~Barbour for suggesting the topic 
and for many helpful discussions.
}

\end{document}